\documentclass[a4paper,12pt]{amsart}
\usepackage[english]{babel}
\usepackage[T1]{fontenc}
\usepackage[left=1in,right=1in,top=1in,bottom=1in]{geometry}
\usepackage{times}
\usepackage{microtype}
\usepackage{appendix}

\usepackage{commath,amssymb,amscd,mathrsfs,mathtools,dsfont,bbm,multirow,array,caption,comment,pifont}
\usepackage[all]{xy}
\usepackage{enumitem}
\usepackage{longtable}

\usepackage{cite}			

\allowdisplaybreaks

\usepackage[usenames,dvipsnames,svgnames,table]{xcolor}
\definecolor{red}{RGB}{255,25,25}
\definecolor{blue}{RGB}{25,50,200}
\usepackage{tikz-cd}
\usepackage[pagebackref,linktocpage]{hyperref}

\hypersetup{
pdftitle={},				
pdfauthor={Fei Hu},	
colorlinks=true,			
linkcolor=red,			
citecolor=MidnightBlue,	
filecolor=magenta,		
urlcolor=MidnightBlue	
}

\usepackage{cleveref}	

\newtheorem{theorem}{Theorem}[section]
\crefname{theorem}{Theorem}{Theorems}
\newtheorem{lemma}[theorem]{Lemma}
\crefname{lemma}{Lemma}{Lemmas}
\newtheorem{proposition}[theorem]{Proposition}
\crefname{proposition}{Proposition}{Propositions}
\newtheorem{prop}[theorem]{Proposition}
\crefname{prop}{Proposition}{Propositions}

\crefname{corollary}{Corollary}{Corollaries}

\crefname{cor}{Corollary}{Corollaries}

\crefname{conjecture}{Conjecture}{Conjectures}

\crefname{conj}{Conjecture}{Conjectures}
\newtheorem*{conj*}{Conjecture}
\crefname{conj}{Conjecture}{Conjectures}

\crefname{conjA}{Conjecture}{Conjecture}

\crefname{conjB}{Conjecture}{Conjecture}

\crefname{conjC}{Conjecture}{Conjecture}

\crefname{conjDk}{Conjecture}{Conjecture}

\crefname{conjD}{Conjecture}{Conjecture}

\crefname{conjH}{Conjecture}{Conjecture}

\crefname{conjGr}{Conjecture}{Conjecture}

\theoremstyle{definition}
\newtheorem{definition}[theorem]{Definition}
\crefname{definition}{Definition}{Definitions}

\crefname{defn}{Definition}{Definitions}

\crefname{example}{Example}{Examples}

\crefname{notation}{Notation}{Notation}
\newtheorem*{notation*}{Notation}
\crefname{notation}{Notation}{Notation}
\newtheorem*{convention*}{Convention}
\crefname{convention}{Convention}{Convention}

\crefname{problem}{Problem}{Problems}
\newtheorem{question}[theorem]{Question}
\crefname{question}{Question}{Questions}

\crefname{condition}{Condition}{Conditions}

\crefname{assumption}{Assumption}{Assumptions}

\crefname{propGr}{Property}{Property}

\theoremstyle{remark}

\crefname{rmk}{Remark}{Remarks}
\newtheorem*{rmk*}{Remark}
\crefname{rmk}{Remark}{Remarks}
\newtheorem{remark}[theorem]{Remark}
\crefname{remark}{Remark}{Remarks}

\crefname{fact}{Fact}{Facts}

\crefname{claim}{Claim}{Claims}
\newtheorem*{claim*}{Claim}
\crefname{claim}{Claim}{Claims}

\crefname{step}{Step}{Steps}

\crefname{case}{Case}{Cases}

\numberwithin{equation}{section}

\newcommand{\what}[1]{\widehat{#1}}

\newcommand{\ol}[1]{\overline{#1}}

\newcommand{\longinjmap}{\lhook\joinrel\longrightarrow}

\newcommand{\lra}{\longrightarrow}

\newcommand{\doi}[1]{\href{https://doi.org/#1}{{\tt doi:#1}}}

\newcommand{\bC}{\mathbf{C}}

\newcommand{\bF}{\mathbf{F}}

\newcommand{\bH}{\mathbf{H}}

\newcommand{\bN}{\mathbf{N}}

\newcommand{\bQ}{\mathbf{Q}}
\newcommand{\bR}{\mathbf{R}}

\newcommand{\bZ}{\mathbf{Z}}

\newcommand{\bbi}{\mathbbm{i}}
\newcommand{\bbj}{\mathbbm{j}}
\newcommand{\bbk}{\mathbbm{k}}
\newcommand{\bk}{\mathbf{k}}

\newcommand{\sF}{\mathscr{F}}

\newcommand{\sL}{\mathscr{L}}

\newcommand{\sO}{\mathscr{O}}

\newcommand{\sT}{\mathsf{T}}

\newcommand{\reduced}{{\rm red}}

\newcommand{\alg}{\operatorname{alg}}

\newcommand{\Char}{\operatorname{char}}

\newcommand{\Div}{\mathsf{Div}}

\newcommand{\End}{\operatorname{End}}
\newcommand{\et}{{\textrm{\'et}}}

\newcommand{\GK}{\operatorname{GKdim}}

\newcommand{\id}{\operatorname{id}}
\newcommand{\im}{\operatorname{Im}}
\newcommand{\isom}{\simeq}

\newcommand{\lov}{\operatorname{lov}}
\newcommand{\Mat}{\operatorname{M}}

\newcommand{\N}{\mathsf{N}}

\newcommand{\num}{\equiv}

\newcommand{\Pic}{\operatorname{Pic}}

\newcommand{\plov}{\operatorname{plov}}

\newcommand{\pr}{\operatorname{pr}}

\newcommand{\re}{\operatorname{Re}}

\newcommand{\Vol}{\operatorname{Vol}}

\makeatletter
\let\@wraptoccontribs\wraptoccontribs
\makeatother

\begin{document}

\title[Polynomial volume growth for abelian varieties]{Polynomial volume growth of quasi-unipotent automorphisms of abelian varieties}

\author{Fei Hu}
\contrib[with an appendix in collaboration with]{Chen Jiang}

\address{
Department of Mathematics, Nanjing University, Nanjing, China
\endgraf
Department of Mathematics, University of Oslo, Oslo, Norway
\endgraf
Department of Mathematics, Harvard University, Cambridge, USA
}

\email{\href{mailto:fhu@nju.edu.cn}{\tt fhu@nju.edu.cn}}

\address{Shanghai Center for Mathematical Sciences \& School of Mathematical Sciences, Fudan University, Shanghai, China}
\email{\href{mailto:chenjiang@fudan.edu.cn}{chenjiang@fudan.edu.cn}}


\begin{abstract}
Let $X$ be an abelian variety over an algebraically closed field $\mathbf{k}$ and $f$ a quasi-unipotent automorphism of $X$.
When $\mathbf{k}$ is the field of complex numbers, Lin, Oguiso, and D.-Q. Zhang provide an explicit formula for the polynomial volume growth of (or equivalently, for the Gelfand--Kirillov dimension of the twisted homogeneous coordinate ring associated with) the pair $(X, f)$, by an analytic argument.
We give an algebraic proof of this formula that works in arbitrary characteristic.

In the course of the proof, we obtain:
(1) a new description of the action of endomorphisms on the $\ell$-adic Tate spaces, in comparison with recent results of Zarhin and Poonen--Rybakov;
(2) a partial converse to a result of Reichstein, Rogalski, and J.J. Zhang on quasi-unipotency of endomorphisms and their pullback action on the rational N\'eron--Severi space $\mathsf{N}^1(X)_{\mathbf{Q}}$ of $\mathbf{Q}$-divisors modulo numerical equivalence;
(3) the maximum size of Jordan blocks of (the Jordan canonical form of) $f^*|_{\mathsf{N}^1(X)_{\mathbf{Q}}}$ in terms of the action of $f$ on the Tate space $V_\ell(X)$.
\end{abstract}



\date{}

\thanks{The author was supported by Young Research Talents grant 300814 from the Research Council of Norway and Start-up grant 14912209 from Nanjing University.}

\maketitle



\section{Introduction}
\label{sec:intro}


Algebraic dynamics with positive entropy has been intensively studied in the last decades (see Oguiso \cite{Oguiso14-ICM} and Dinh--Sibony \cite{DS17} for surveys).
Recently, there has also been a growing interest in dynamics with zero entropy.
We refer to Cantat \cite[\S 3]{Cantat18-ICM}, Cantat--Paris-Romaskevich \cite{CPR21}, Fan--Fu--Ouchi \cite{FFO21}, Dinh--Lin--Oguiso--Zhang \cite{DLOZ22}, and references therein for the state-of-the-art on this subject.

Given a surjective self-morphism $f$ of a smooth complex projective variety $X$ of dimension $d$, Gromov \cite{Gromov03} introduced in 1977 the so-called iterated graphs $\Gamma_n \subset X^{n}$ of $f$ and bounded the topological entropy $h_{\mathrm{top}}(f)$ of $f$ by the {\it volume growth $\lov(f)$} of $f$ and then by the {\it algebraic entropy $h_{\mathrm{alg}}(f)$} of $f$ as follows:
\begin{align*}
h_{\mathrm{top}}(f) &\leq \lov(f) \coloneqq \limsup_{n\to \infty} \frac{\log \Vol (\Gamma_n)}{n} \\
&\leq h_{\alg}(f) \coloneqq \log \max_{0\leq k\leq d} \lambda_k(f),
\end{align*}
where the volume of $\Gamma_n$ is computed against the ample divisor on the product variety $X^{n}$ induced from an arbitrary ample divisor $H_X$ on $X$ and the $k$-th dynamical degree $\lambda_k(f)$ of $f$ is defined as the spectral radius of the pullback action $f^*|_{H^{k,k}(X)}$ of $f$ on the Dolbeault cohomology group $H^{k,k}(X)$.
This, in conjunction with Yomdin's remarkable inequality $h_{\mathrm{top}}(f) \geq h_{\mathrm{alg}}(f)$ (which resolves Shub's entropy conjecture; see \cite{Yomdin87}), yields the fundamental equality in higher-dimensional algebraic/holomorphic dynamics, saying that
\[
h_{\mathrm{top}}(f) = \lov(f) = h_{\mathrm{alg}}(f).
\]

\subsection{Polynomial volume growth of automorphisms of zero entropy}
\label{subsec:LOZ}

Given an algebraic dynamical system $(X, f)$ with zero entropy,
equivalently, the pullback action $f^*|_{\N^1(X)_\bQ}$ of $f$ on the rational N\'eron--Severi space $\N^1(X)_\bQ \coloneqq \Div(X)_\bQ/\!\num$ of numerical classes of $\bQ$-divisors is {\it quasi-unipotent} (i.e., all of its eigenvalues are roots of unity; see \cref{lemma:zero-entropy}),
there is no exponential growth of the volumes of iterated graphs.
However, one may consider polynomial growth.
Indeed, Cantat and Paris-Romaskevich \cite{CPR21} introduced the so-called polynomial volume growth $\plov(f)$ of $f$ as follows:
\[
\plov(f) \coloneqq \limsup_{n\to \infty} \frac{\log \Vol (\Gamma_n)}{\log n},
\]
which is also well defined in arbitrary characteristic (see \S\ref{sec:plov} for details).
In \cite{LOZ}, Lin, Oguiso, and Zhang found that for an automorphism $f$ of zero entropy of $X$, the invariant $\plov(f)$ turns out to be equivalent to the Gelfand--Kirillov dimension $\GK B$ of the twisted homogeneous coordinate ring $B$ associated with $(X, f)$, as a direct consequence of Keeler's earlier work \cite{Keeler00} on $f$-ample\footnote{In the literature \cite{AVdB90,Keeler00,RRZ06}, it was named as $\sigma$-ampleness where $\sigma$ is an automorphism of $X$. We prefer to use $f$ to denote an automorphism for consistency and hope this will not cause any confusion.} divisors.
Among other things, they proved an explicit formula computing $\plov(f)$ for complex tori (see \cite[Theorem~1.4]{LOZ}).

Note that $\plov(f)$ and $\GK B$ are both defined algebraically via intersection theory, so are still meaningful in positive characteristic.
Inspired by earlier related works \cite{Hu19,Hu-lc}, we are thus wondering if their formula holds for abelian varieties in positive characteristic.
Our first result gives an affirmative answer to this question.
Let us fix some notations.

From now on till the end of this article, we will be working over an algebraically closed field $\bk$ of arbitrary characteristic.
We refer to Esnault--Srinivas\cite{ES13} for the famous pioneering work on algebraic dynamics in positive characteristic.
Let $X$ be an abelian variety of dimension $g$ over $\bk$.
Let $\End(X)$ be the endomorphism ring of $X$.
Denote by
\[
\End^\circ(X) \coloneqq \End(X)\otimes_\bZ \bQ
\]
the semisimple $\bQ$-algebra of endomorphisms (with rational coefficients) of $X$.
Let $\ell$ be a prime different from $\Char(\bk)$ and $T_\ell(X) \coloneqq \varprojlim_{n} X[\ell^n](\bk)$ be the $\ell$-adic Tate module of $X$, which is a free $\bZ_\ell$-module of rank $2g$.
We call the corresponding $\bQ_\ell$-vector space $V_\ell(X) \coloneqq T_\ell(X) \otimes_{\bZ_\ell} \bQ_\ell$ the $\ell$-adic Tate space of $X$.
This gives us the following faithful $\ell$-adic representation
\begin{equation*}
V_\ell \colon \End^\circ(X) \otimes_\bQ \bQ_\ell \longinjmap \End_{\bQ_\ell}(V_\ell(X)) = \Mat_{2g}(\bQ_\ell).
\end{equation*}

\begin{theorem}
\label{thm:B}
Let $f$ be an automorphism of an abelian variety $X$ of dimension $g$ over $\bk$.
Suppose that $f$ is of zero entropy (see \cref{lemma:zero-entropy}).
Denote the Jordan canonical form of $V_\ell(f)$ by $J \oplus \ol{J}$, where $J = \bigoplus_i J_{\zeta_i,k_i} \in \Mat_{g}(\ol\bQ)$ for some roots of unity $\zeta_i\in \ol\bQ$ and positive integers $k_i\in\bZ_{>0}$.
Then the polynomial volume growth $\plov(f)$ of $f$ is equal to $\sum_i k_i^2$.
\end{theorem}

One may wonder in \cref{thm:B} why the Jordan canonical form of $V_\ell(f)$ has to be of the form $J \oplus \ol{J}$ for $J\in \Mat_{g}(\ol\bQ)$, since, a priori, $V_\ell(f)$ lies in $\Mat_{2g}(\bQ_\ell)$.
This is actually our next main \cref{thm:A} on the induced action of endomorphisms on the $\ell$-adic Tate spaces.

\begin{remark}
As mentioned earlier, \cref{thm:B} has been proved by Lin, Oguiso, and Zhang for complex tori (see \cite[Theorem~1.4]{LOZ}).
Their proof is analytic and relies on calculations of differential forms; in particular, certain positivity of $(1,1)$-forms is involved.
But so far, we are not aware of any such positivity in \'etale cohomology.
Secondly, following \cite{AVdB90,Keeler00}, they also decompose the unipotent action $f^*|_{\N^1(X)_\bQ}$ as $\id + N$ for some nilpotent operator $N$ on $\N^1(X)_\bQ$ satisfying that $N^{k(f)}\neq 0$ and $N^{k(f)+1}=0$, i.e., $k(f)+1$ is the maximum size of Jordan blocks of (the Jordan canonical form of) $f^*|_{\N^1(X)_\bQ}$.
We shall discuss $k(f)$ in more details in \S\ref{subsec:DLOZ} later.
The advantage of this decomposition is that for {\it any} $n$, the volume $\Vol (\Gamma_n)$ of the iterated graph $\Gamma_n\subset X^n$ becomes the top self-intersection of a divisor
\[
\Delta_n\coloneqq \Delta_n(f, H_X) \coloneqq \sum_{m=0}^{n-1} (f^m)^*H_X = \sum_{i=0}^{k(f)} \binom{n}{i+1} N^iH_X
\]
on $X$, which is only a finite sum of divisors of the form $N^iH_X$.
However, the nilpotent operator $N$ is not geometric in nature so that the (non)vanishing of intersections of divisors of the form $N^iH_X$ does not seem to be easy to determine (see \cref{rmk:plov}).
\end{remark}

\begin{remark}
Our proof of \cref{thm:B} is algebraic and characteristic-free, though the idea is simple.
We interpret the top self-intersection number of $\Delta_n$ as the degree of the endomorphism $\phi_{\sO_X(\Delta_n)}$ induced by the line bundle $\sO_X(\Delta_n)$ on $X$ using the Riemann--Roch theorem for abelian varieties.
Then based on previous works of Keeler \cite{Keeler00} and Lin--Oguiso--Zhang \cite{LOZ}, it suffices to evaluate the degree of the degree function $\deg(\phi_{\sO_X(\Delta_n)})$ as a polynomial in $n$, which can be done by an explicit calculation with the aid of \cref{thm:A,thm:C,prop:A*A}.
\end{remark}

\subsection{Pseudo-analytic representation of endomorphism algebras}
\label{subsec:ZPR}

As mentioned earlier, the main motivation of this subsection is its application to the polynomial volume growth for abelian varieties in arbitrary characteristic discussed in \S\ref{subsec:LOZ}.
Besides, another one is the following very recent results on certain rational or integral structure of the induced $\ell$-adic representation of $\End^\circ(X)$ on the Tate modules/spaces.

\begin{theorem}[{Zarhin \cite[Theorem~1.1]{Zarhin20} and Poonen--Rybakov \cite[Theorem~1.2 (b)]{PR21}}]
\label{thm:ZPR}
Let $X$ be an abelian variety of dimension $g$ over $\bk$.
\begin{enumerate}[label=\emph{(\arabic*)}, ref=\arabic*]
\item For any $\alpha\in \End^\circ(X)$, there is a matrix $M \in \Mat_{2g}(\bQ)$ such that there is a $\bQ_\ell$-basis of $V_\ell(X)$ on which the induced action of $\alpha$ is represented by $M$.
\item For any $\alpha\in \End(X)$, there is a matrix $M \in \Mat_{2g}(\bZ)$ such that there is a $\bZ_\ell$-basis of $T_\ell(X)$ on which the induced action of $\alpha$ is represented by $M$.
\end{enumerate}
\end{theorem}

Clearly, the characteristic zero case of \cref{thm:ZPR} can be reduced to the complex case, and then one just takes rational or integral homology (see \cite[Remark~1.2]{Zarhin20}).
Moreover, for a smooth complex projective variety $Y$, its cohomology groups $H^*(Y, \bZ)$ carry Hodge structure; in particular, for any self-morphism $f$ of $Y$, the induced action $f^*|_{H^{1}(Y, \bC)}$ has the following decomposition:
\begin{align*}
f^*|_{H^{1}(Y, \bC)} = f^*|_{H^{1,0}(Y)} \oplus f^*|_{H^{0,1}(Y)} = f^*|_{H^{1,0}(Y)} \oplus \ol{f^*|_{H^{1,0}(Y)}}.
\end{align*}
This forces all eigenvalues of $f^*|_{H^{1}(Y, \bC)}$ have to appear in pairs.
In an earlier work on the dynamical degree comparison conjecture, the author proved a positive characteristic analog of this parity phenomena for abelian varieties (see \cite[Theorem~1.1]{Hu-lc}).

Inspired by Zarhin, Poonen, and Rybakov's recent works on endomorphisms of abelian varieties, we now slightly strengthen our previous result to the following form, which describes a new property of the $\ell$-adic representation $V_\ell$ of endomorphism algebras.

\begin{theorem}
\label{thm:A}
Let $X$ be an abelian variety of dimension $g$ over $\bk$ and $\alpha\in \End^\circ(X)$.
Then the Jordan canonical form of $V_\ell(\alpha)$ is of the form $J\oplus \ol{J}$, where $J\in \Mat_g(\ol\bQ)$ is a direct sum of some standard Jordan blocks and $\ol{J}$ is its complex conjugate.
\end{theorem}

We thus introduce the following notion.

\begin{definition}
\label{def:pseudo-analytic}
Let $F$ be a field with an algebraic closure $\ol{F} \subseteq \bC$.
A matrix $M \in \Mat_{2n}(F)$ admits a {\it pseudo-analytic decomposition}, if $M$ is similar to $J \oplus \ol{J}$, where $J\in \Mat_n(\ol{F})$ is a direct sum of some standard Jordan blocks and $\ol{J}$ is its complex conjugate.
\end{definition}

\begin{remark}
Our \cref{thm:A} is independent of Zarhin \cite[Theorem~1.1]{Zarhin20} or Poonen--Rybakov \cite[Theorem~1.2 (b)]{PR21}, though it seems that we need to apply their results first so that it is possible to speak of a pseudo-analytic decomposition of $V_\ell(\alpha)$.
In fact, by a result of Milne (cf.~\cref{prop:Milne}), we see that $V_\ell(\alpha)$ is actually defined over $\ol\bQ$.
Based on a similar idea in the proof of \cite[Theorem~1.1]{Hu-lc}, we reduce the problem to the study of the (reduced) representation of semisimple $\bQ$-algebras of endomorphisms of abelian varieties.
On the other hand, our result does not seem to imply theirs either.
\end{remark}

\subsection{Quasi-unipotent action of endomorphisms of abelian varieties}
\label{subsec:RRZ}

Let $\alpha\in \End(X)$ be an endomorphism of an abelian variety $X$ of dimension $g$ over $\bk$.
Besides the induced action $V_\ell(\alpha)$ of $\alpha$ on the Tate space $V_\ell(X)$ discussed in \S\ref{subsec:ZPR}, it is also natural to consider its pullback action $\alpha^*|_{\N^1(X)_\bQ}$ on the rational N\'eron--Severi space $\N^1(X)_\bQ$ of $\bQ$-divisors modulo numerical equivalence:
\[
\begin{array}{cccc}
\nu\colon \End(X) & \lra & \End_\bQ(\N^1(X)_\bQ) \isom \Mat_{\rho}(\bQ) \\[2pt]
\alpha & \longmapsto & \nu(\alpha) \coloneqq \alpha^*|_{\N^1(X)_\bQ},
\end{array}
\]
where $\rho\coloneqq \dim_\bQ \N^1(X)_\bQ$ denotes the Picard number of $X$.
One can naturally extend $\nu$ to $\End^\circ(X) \to \End_\bQ(\N^1(X)_\bQ)$ by sending $\alpha\in\End^\circ(X)$ to $n^{-2}(n\alpha)^*|_{\N^1(X)_\bQ}$, if $n\alpha\in \End(X)$ for some $n\in \bZ_{>0}$.
We still denote it by $\nu$.
Note that $\nu$ is merely a homomorphism of multiplicative semigroups (or rather, monoids), but {\it not} a ring homomorphism (i.e., $\nu(\alpha + \beta) \neq \nu(\alpha) + \nu(\beta)$, in general).
Nonetheless, in their study of projective simplicity of the twisted homogeneous coordinate rings, Reichstein, Rogalski, and Zhang proved that $\nu$ preserves (quasi)-unipotency\footnote{Recall that for a ring $R$ with identity, an element $r\in R$ is {\it unipotent} if $r-1$ is nilpotent and {\it quasi-unipotent} if some power of $r$ is unipotent.}.
Precisely, if $\alpha \in \End^\circ(X)$ is (quasi-)unipotent, then so is $\nu(\alpha) \in \Mat_{\rho}(\bQ)$; see \cite[Theorem~8.5 (a)]{RRZ06}.
We obtain a partial converse to this as follows.


\begin{theorem}
\label{thm:C}
With notation as above, let $\alpha \in \End(X)$.
Then the following statements are equivalent.
\begin{enumerate}[label=\emph{(\arabic*)}, ref=\arabic*]
\item \label{thm:C-1} $\alpha \in \End(X)$ is quasi-unipotent.
\item \label{thm:C-2} $\nu(\alpha) = \alpha^*|_{\N^1(X)_\bQ} \in \Mat_{\rho}(\bQ)$ is quasi-unipotent.
\item \label{thm:C-3} $\alpha^*|_{H^2_{\emph{\et}}(X, \bQ_\ell)} \in \Mat_{b_2}(\bQ_\ell)$ is quasi-unipotent.
\item \label{thm:C-4} $V_\ell(\alpha) = \alpha|_{V_\ell(X)} \in \Mat_{2g}(\bQ_\ell)$ is quasi-unipotent.
\end{enumerate}
\end{theorem}

\begin{remark}
\label{rmk:C}
Let $\alpha \in \End(X)$ and $P_\alpha(t)\in \bZ[t]$ denote the characteristic polynomial of $\alpha$ (see \cref{def:char-poly}).
Suppose that $\alpha^*|_{\N^1(X)_\bQ}$ is quasi-unipotent.
Then by \cref{thm:C}, we see that all roots of $P_\alpha(t)$ are roots of unity.
Hence $\deg(\alpha) = P_\alpha(0) = 1$ and $\alpha$ turns out to be an automorphism.
On the other hand, one cannot hope the converse of \cite[Theorem~8.5 (a)]{RRZ06} for unipotency by just considering $\alpha = [-1]_X$.

However, the situation becomes a bit subtle if one considers $\alpha \in \End^\circ(X)$ which is not a genuine endomorphism.
See \cref{rmk:C-End0} for details.
\end{remark}

\subsection{Index of unipotent automorphisms acting on the N\'eron--Severi space}
\label{subsec:DLOZ}

In this subsection, we consider a more delicate structure on automorphisms $f$ of zero entropy (i.e., $f^*|_{\N^1(X)_\bQ}$ is quasi-unipotent), the maximum size of Jordan blocks of (the Jordan canonical form of) $f^*|_{\N^1(X)_\bQ}$, denoted by $k(f)+1$.
It seems to first appear in \cite[Lemma~5.4]{AVdB90} when Artin and Van den Bergh studied the GK-dimension of the twisted homogeneous coordinate rings of surfaces; precisely, they showed that $k(f)$ is either $0$ or $2$ for surfaces.
It also plays an essential role in Keeler's work \cite[Theorem~6.1]{Keeler00} in higher dimensions.
Cantat also obtained this result independently (see the appendix to \cite{DF01} or \cite[Theorem~2.13]{Cantat14a}).
Note that $k(f)$ is always an even natural number by \cite[Lemma~6.12]{Keeler00}.
Recently, Dinh, Lin, Oguiso, and Zhang \cite{DLOZ22} proved the following striking result, extending Cantat \cite{Cantat14a} and Lo Bianco \cite{LB19} to higher dimensions.

\begin{theorem}[{cf.~\cite[Theorem~1.1]{DLOZ22}}]
\label{thm:DLOZ}
Let $X$ be a smooth complex projective variety of dimension $d$.
Let $f$ be an automorphism of zero entropy of $X$.
Then the maximum size of Jordan blocks of (the Jordan canonical form of) $f^*|_{\N^1(X)_\bQ}$ is at most $2d-1$, i.e., $k(f)\leq 2(d-1)$.
\end{theorem}

\begin{remark}
We remark that the original \cite[Theorem~1.1]{DLOZ22} is stated for automorphisms $f$ of zero entropy of compact K\"ahler manifolds $X$ and concerns the polynomial growth of norms of iterates $f^n$ acting on the Dolbeault cohomology $H^{1,1}(X)$:
\[
\|(f^n)^*|_{H^{1,1}(X)}\| = O(n^{2(d-1)}),
\]
as $n$ tends to infinity.
Note that $\|(f^n)^*|_{H^{1,1}(X)}\| = O(n^k)$ if and only if the maximum size of Jordan blocks of $f^*|_{H^{1,1}(X)}$ is at most $k+1$.
It is easy to see, in the complex projective setting, that the polynomial growth of $f^n$ acting on $\N^1(X)_\bQ$ and on $H^{1,1}(X)$ are equivalent.
Their proof is based on quasi-nef sequences introduced by Zhang \cite{Zhang09b} and the associated dynamical filtrations on $H^{1,1}(X, \bR)$ (see \cite[\S 3]{DLOZ22}).
\end{remark}

As another application of our \cref{thm:A}, we prove a positive characteristic analog of \cref{thm:DLOZ} for abelian varieties as follows.

\begin{theorem}
\label{thm:D}
Let $f$ be an automorphism of an abelian variety $X$ of dimension $g$ over $\bk$.
Suppose that $f$ is of zero entropy.
Denote the Jordan canonical form of $V_\ell(f)$ by $J \oplus \ol{J}$, where $J$ is a direct sum of standard Jordan blocks with maximum size $k(J)+1$.
Then the maximum size of Jordan blocks of $f^*|_{\N^1(X)_\bQ}$ is $2k(J)+1$.
In particular, we have that 
\[
\|(f^n)^*|_{\N^1(X)_\bQ}\| \sim n^{k(f)} = n^{2k(J)} = O(n^{2(g-1)}),
\]
as $n$ tends to infinity.
\end{theorem}

\begin{convention*}
\label{convention}
Throughout, for two positive sequences $a_n$ and $b_n$, the notation $a_n \sim b_n$ means that $\lim_{n\to \infty} a_n/b_n = C$ for some (unspecified) positive constant $C>0$. Note that, by allowing $C\neq 1$, our usage is more general than the common one in the literature.
On the other hand, the same symbol $\sim$ is also used to denote similarity of matrices.
\end{convention*}

\begin{remark}
The authors of \cite{DLOZ22} have also carried out a similar calculation when $f$ is an automorphism of a complex torus $X=\bC^g/\Lambda$.
In that case, if $f^*|_{H^{1,0}(X)}$ is represented by a standard Jordan block $J_{1,g}$ with eigenvalue $1$ and of size $g$, then $f^*|_{H^{1,1}(X)}$ could be represented by the Kronecker product $J_{1,g}\otimes J_{1,g}$ of $J_{1,g}$ (see \cite[\S 4.2]{HJ94}).
It follows immediately that
\[
\|(f^n)^*|_{H^{1,1}(X)}\| = \| (J_{1,g}\otimes J_{1,g})^n \| = \| J_{1,g}^n\otimes J_{1,g}^n \| = \| J_{1,g}^n \|^2 \sim n^{2(g-1)},
\]
as $n$ tends to infinity.
See \cite[\S 4.1]{DLOZ22}.
However, in positive characteristic, this argument does not work straightforwardly.
Instead, we use the fact that $H^2_{\et}(X, \bQ_\ell)$ is the wedge product of $H^1_{\et}(X, \bQ_\ell)$ and the induced action of $f$ on $H^2_{\et}(X, \bQ_\ell)$ is represented by the second compound\footnote{The {\it $r$-th compound matrix $C_r(M)$} of a given matrix $M\in \Mat_{n}(\bF)$ over some field $\bF$ is defined as the matrix of all $r$-by-$r$ minors of $M$;
see, e.g., \cite[\S 6.6]{KRY09}.} matrix of $f^*|_{H_{\et}^1(X,\bQ_\ell)}$.

In the end, we remark that, in order to study dynamics of automorphisms of zero entropy of arbitrary smooth projective varieties $X$ in positive characteristic, it might also be interesting to define dynamical filtrations on $\N^1(X)_\bR$ based on \cite{Hu20a}, though it will not be discussed here.
\end{remark}


\section{Polynomial volume growth of automorphisms of zero entropy}
\label{sec:plov}


We work over an algebraically closed field $\bk$ of arbitrary characteristic.
Let $X$ be a smooth projective variety of dimension $d$, $H_X$ an ample divisor on $X$, and $f\colon X\to X$ a surjective self-morphism of $X$, all defined over $\bk$.
For any prime $\ell \neq \Char(\bk)$, we denote $\ell$-adic \'etale cohomology by $H^*_{\et}(X, \bQ_\ell)$.
For any $0\leq k\leq d$, let $\N^k(X)$ denote the group of (integral) algebraic cycles modulo numerical equivalence so that $\N^k(X)_\bR \coloneqq \N^k(X) \otimes_\bZ \bR$ becomes a finite-dimensional real vector space.
We choose an arbitrary norm $\|\cdot\|$ on it.
Then the {\it $k$-th (numerical) dynamical degree $\lambda_k(f)$} of $f$ is defined by
\begin{align*}
\lambda_k(f) &\coloneqq \lim_{n\to \infty} ((f^n)^*H_X^k \cdot H_X^{d-k})^{1/n} \\
&= \lim_{n\to \infty} \|(f^n)^*|_{\N^k(X)_\bR}\|^{1/n} \\
&= \rho(f^*|_{\N^k(X)_\bR}),
\end{align*}
where $\rho(\varphi)$ denotes the spectral radius of a linear transformation $\varphi$.
It is well known that the $\lambda_k(f)$ satisfy the log-concavity property following from the Khovanskii--Teissier inequality (see, e.g., \cite[Lemma~4.7]{HT}).
In particular, $\lambda_1>1$ if and only if $\lambda_k(f)>1$ for all (or for some) $k$.

\begin{definition}
\label{def:entropy}
The {\it algebraic entropy $h_{\alg}(f)$} of $f$ is defined by
\[
h_{\alg}(f) \coloneqq \log \max_{0\leq k\leq d} \lambda_k(f).
\]
We say that $f$ is of {\it zero entropy}, if $h_{\alg}(f) = 0$; otherwise, $f$ is of {\it positive entropy}.
\end{definition}

We include the following standard facts for the convenience of the reader.

\begin{lemma}
\label{lemma:zero-entropy}
With notation as above, the following statements are equivalent.
\begin{enumerate}[label=\emph{(\arabic*)}, ref=\arabic*]
\item \label{lemma:zero-entropy-1} $f$ is of zero entropy.
\item \label{lemma:zero-entropy-2} All dynamical degrees $\lambda_k(f)$ of $f$ are equal to one.
\item \label{lemma:zero-entropy-3} The first dynamical degree $\lambda_1(f)$ of $f$ is equal to one.
\item \label{lemma:zero-entropy-4} All eigenvalues of $f^*|_{\N^1(X)_\bQ}$ are roots of unity.
\item \label{lemma:zero-entropy-5} The pullback action $f^*|_{\N^1(X)_\bQ}$ of $f$ on $\N^1(X)_\bQ$ is {\it quasi-unipotent}, i.e., $(f^n)^*|_{\N^1(X)_\bQ}$ is unipotent for some $n\in \bZ_{>0}$.
\end{enumerate}
\end{lemma}
\begin{proof}
Clearly, Statement~\eqref{lemma:zero-entropy-1} $\Longleftrightarrow$ Statement~\eqref{lemma:zero-entropy-2} $\Longrightarrow$ Statement~\eqref{lemma:zero-entropy-3} and Statement~\eqref{lemma:zero-entropy-5} $\Longleftrightarrow$ Statement~\eqref{lemma:zero-entropy-4} $\Longrightarrow$ Statement~\eqref{lemma:zero-entropy-3} hold by definition.
Statement~\eqref{lemma:zero-entropy-3} $\Longrightarrow$ Statement~\eqref{lemma:zero-entropy-2} follows from the log-concavity property of the dynamical degrees.
Statement~\eqref{lemma:zero-entropy-3} $\Longrightarrow$ Statement~\eqref{lemma:zero-entropy-4} follows from the fact that all eigenvalues of $f^*|_{\N^1(X)_\bQ}$ are algebraic integers and a classical theorem of Kronecker asserting that a nonzero algebraic integer $\alpha$ is a root of unity if all of its Galois conjugates (including $\alpha$) have (complex) absolute value at most $1$.
\end{proof}

\begin{remark}
We refer to \cite[Definition~1.1 and Conjecture~1.3]{HT} for another definition of (cohomological) dynamical degrees using $\ell$-adic \'etale cohomology and the so-called dynamical degree comparison (DDC) conjecture (proposed by Truong \cite{Truong}), respectively.
Among other things, Truong proved equivalence of dynamical degrees when $f$ is of zero entropy, and hence it is unnecessary to distinguish these two definitions of dynamical degrees.
Besides some partial results on the abelian variety case \cite{Hu19,Hu-lc}, the DDC conjecture is still wide open since it would surprisingly imply a generalization of Weil's Riemann Hypothesis for polarized endomorphisms (see \cite[Theorem~1.11 (3)]{HT}).
\end{remark}

We now study algebraic dynamics of an automorphism $f$ of zero entropy of a smooth projective variety $X$ of dimension $d$ over $\bk$.
We give the definition of the polynomial volume growth $\plov(f)$ of $f$ and investigate some basic properties.
We refer to \cite{CPR21,DLOZ22,LOZ} for recent advance on this subject (especially, when $\bk=\bC$).

Recall that Gromov's iterated graph $\Gamma_n\subset X^n$ is the graph of the morphism
\[
\begin{array}{cccc}
(f,\ldots,f^{n-1}) \colon & X & \lra & X^{n-1} \\[2pt]
& x & \longmapsto & (f(x), f^2(x), \ldots, f^{n-1}(x)),
\end{array}
\]
for any $n\in \bZ_{>0}$.
The volume of $\Gamma_n$ with respect to a fixed ample divisor $H_X$ on $X$ (or rather, the induced ample divisor on the product variety $X^{n}$) is defined as follows:
\[
\Vol(\Gamma_n) = \Gamma_n \cdot \Bigg(\sum_{i=1}^{n}\pr_i^*\!H_X\Bigg)^d = \Delta_n(f, H_X)^d,
\]
where $\pr_i \colon X^n \to X$ is the projection from $X^n$ to its $i$-th factor $X$ and
\[
\Delta_n(f, H_X) \coloneqq \sum_{m=0}^{n-1} (f^m)^*H_X.
\]

Cantat and Paris-Romaskevich \cite{CPR21} introduce the polynomial volume growth\footnote{Note, however, that they denote it by $\mathrm{povol}(f)$. In this article, we follow the notation in \cite{Gromov03,LOZ}.} $\plov(f)$ of $f$ when they study the so-called polynomial entropy $h_{\mathrm{pol}}(f)$ of $f$ in holomorphic dynamics.
Here we present an algebraic variant.
The most fascinating point of this dynamical invariant $\plov(f)$ is its connection with the so-called Gelfand--Kirillov dimension $\GK B$ of the twisted homogeneous coordinate ring $B$ associated with $(X, f, \sO_X(H_X))$.
This surprising connection is first noticed by Lin, Oguiso, and Zhang; see \cite[\S 8]{LOZ}.

\begin{definition}[{cf. \cite[eq.~(2.7), p.~1354]{CPR21}}]
\label{def:plov}
With notation as above, the {\it polynomial volume growth $\plov(f)$} of $f$ (with respect to $H_X$) is defined by
\[
\plov(f) \coloneqq \limsup_{n\to \infty} \frac{\log \Vol (\Gamma_n)}{\log n} = \limsup_{n\to \infty} \frac{\log (\Delta_n(f, H_X)^d)}{\log n}.
\]
\end{definition}

Many basic properties of the above invariant $\plov$ have been summarized in \cite[\S 2.1]{LOZ}.
We only collect three of them, whose proofs are the same as \cite{LOZ}.

\begin{lemma}[{cf.~\cite[Lemma~2.1]{LOZ}}]
\label{lemma:plov-indep-div}
In \cref{def:plov}, the invariant $\plov(f)$ is independent of the choice of ample $H_X$.
\end{lemma}

The result below, also due to Lin--Oguiso--Zhang, says that $\plov$ is invariant after iterations.

\begin{lemma}[{cf.~\cite[Lemma~2.6]{LOZ}}]
\label{lemma:plov-iterate}
For any $0\neq m\in \bZ$, we have $\plov(f^m) = \plov(f)$.
\end{lemma}

By the above lemma, when studying $\plov(f)$ of an automorphism $f$ of zero entropy of $X$, it is harmless to replace $f$ by some iterate so that $f^*|_{\N^1(X)_\bQ}$ is {\it unipotent}, i.e., all eigenvalues of $f^*|_{\N^1(X)_\bQ}$ are $1$.
Then following \cite{AVdB90,Keeler00}, one can decompose $f^*|_{\N^1(X)_\bQ}$ as $\id + N$ for some nilpotent operator $N$ on $\N^1(X)_\bQ$ satisfying that $N^{k(f)}\neq 0$ and $N^{k(f)+1}=0$, i.e., $k(f)+1$ is the maximum size of Jordan blocks of $f^*|_{\N^1(X)_\bQ}$ (a.k.a. the index of the eigenvalue $1$ of $f^*|_{\N^1(X)_\bQ}$).\footnote{Note that $k(f)$ is always an even natural number; see \cite[Lemma~6.12]{Keeler00}.}
It thus follows that
\begin{align*}
\Delta_n(f, H_X) &\num \sum_{m=0}^{n-1} (\id + N)^m H_X = \sum_{m=0}^{n-1} \sum_{i=0}^{m} \binom{m}{i} N^i H_X \\
&= \sum_{i=0}^{n-1} \sum_{m=i}^{n-1} \binom{m}{i} N^i H_X = \sum_{i=0}^{k(f)} \binom{n}{i+1} N^i H_X.
\end{align*}
Using this decomposition, the authors of \cite{LOZ} have obtained the following finiteness result on $\plov$.
See also \cite[Proof of Lemma~6.13]{Keeler00}.

\begin{lemma}[{cf.~\cite[Lemma~2.16]{LOZ}}]
\label{lemma:plov}
With notation as above, $\plov(f)$ is equal to the degree of the following polynomial (with indeterminate $n$)
\begin{align*}
\Delta_n(f, H_X)^d &= \Bigg( \sum_{i=0}^{k(f)} \binom{n}{i+1} N^iH_X \Bigg)^d \\
&=\sum_{0\leq i_1,\ldots,i_d \leq k(f)} \, \prod_{j=1}^d \binom{n}{i_j+1} N^{i_j}H_X.
\end{align*}
In particular, one has that
\[
\plov(f) \leq d+\max\Bigg\{ \sum_{j=1}^d i_j : N^{i_1}H_X \cdots N^{i_d} H_X \neq 0, \, 0\leq i_1,\ldots,i_d \leq k(f) \Bigg\}.
\]
\end{lemma}

\begin{remark}
\label{rmk:plov}
It turns out to be a nontrivial question to determine vanishing or nonvanishing of the intersection numbers $N^{i_1}H_X \cdots N^{i_d} H_X$ in general.
Using dynamical filtrations introduced in \cite{DLOZ22}, Lin, Oguiso, and Zhang manage to obtain some nontrivial vanishing results of this type and derive an upper bound for $\plov$ in characteristic zero.
See \cite[Theorem~4.2]{LOZ}, which improves Keeler's upper bound given in \cite[Theorem~6.1 (3)]{Keeler00}.
\end{remark}

In a separate note, via a combinatorial approach, the author manages to prove the following.

\begin{theorem}
\label{thm:quadratic}
Let $f$ be an automorphism of a smooth projective variety $X$ of dimension $d\geq 2$ over $\bk$.
Suppose that $f^*|_{\N^1(X)_\bQ}$ is (quasi-)unipotent with $k(f)=2$.
Then we have
\[
\plov(f) \leq 2\lfloor \frac{d}{2} \rfloor + d = \left\{
\begin{array}{ll}
2d, & \text{for even } d, \\
2d-1, & \text{for odd } d.
\end{array}
\right.
\]
\end{theorem}

This upper bound is optimal (see \cref{rmk:optimal}) and extends Keeler's upper bound $3d-2$ and Lin--Oguiso--Zhang's upper bound $3d-4$ when $k(f)=2$.
Note, however, that both of their upper bounds are for all general (even) $k(f)$.
During the preparation of this paper, we were informed by the authors of \cite{LOZ} that they also obtained \cref{thm:quadratic} using dynamical filtrations.
Based on our numerical calculations for small $d$ and $k(f)$ and the unimodality of restricted partition numbers, we ask the following.

\begin{question}
\label{conj:plov-upper-bound}
Let $f$ be an automorphism of a smooth projective variety $X$ of dimension $d$ over $\bk$.
Suppose that $f^*|_{\N^1(X)_\bQ} = \id + N$ is unipotent with $N^{k(f)}\neq 0$ and $N^{k(f)+1}=0$.
Then do we have that for any $0\leq i_1,\ldots,i_d \leq k(f)$ satisfying $\sum_{j=1}^d i_j > dk(f)/2$,
\[
N^{i_1}H_X \cdots N^{i_d} H_X = 0?
\]
\end{question}

Note that if \cref{conj:plov-upper-bound} is true, then by \cref{lemma:plov} one immediately obtains that $\plov(f) \leq d + dk(f)/2$, which in the complex case is further bounded by $d^2$ by \cite[Theorem~1.1]{DLOZ22}.
In a recent joint work \cite{HJ}, Chen Jiang and the author give an affirmative answer to \cref{conj:plov-upper-bound} and hence to \cite[Question~4.1]{CPR21} and \cite[Question~1.5 (1)]{LOZ} too.

\section{Endomorphism algebras of abelian varieties: Proof of Theorem \ref{thm:A}}
\label{sec:AV}


In this section, we will focus on the main object of this note -- abelian varieties.
We refer to Mumford's book \cite{Mumford} and Milne \cite{Milne86} for standard notations and terminologies on them.
For the convenience of the reader, we first give a brief review of abelian varieties, including: endomorphism algebras, characteristic polynomials, and the Riemann--Roch theorem.
We then prove \cref{thm:A} using some basic representation theory of semisimple algebras.

Let $X$ be an abelian variety of dimension $g$ over an algebraically closed field $\bk$.
The dual abelian variety $\Pic^\circ(X)$ of $X$ is denoted by $\what X$.
For any line bundle $\sL$ on $X$, there is an induced homomorphism
\[
\phi_\sL\colon X \lra \what X, \quad x \mapsto t_x^*\sL \otimes \sL^{-1}.
\]
We fix an ample divisor $H_X$ on $X$ so that $\phi\coloneqq \phi_{\sO_X(H_X)}$ is an isogeny.
 
Let $\End(X)$ be the endomorphism ring of $X$.
Denote by
\[
\End^\circ(X) \coloneqq \End(X)\otimes_\bZ \bQ
\]
the semisimple $\bQ$-algebra of endomorphisms of $X$.
The {\it Rosati involution} $^\dagger$ on $\End^\circ(X)$ (with respect to the fixed $H_X$) is defined as follows: for any $\alpha \in \End^\circ(X)$,
\[
\alpha^\dagger \coloneqq \phi^{-1}\circ \what \alpha \circ \phi \in \End^\circ(X),
\]
where $\what \alpha\in \End^\circ(\what X)$ is the dual endomorphism (with $\bQ$-coefficients) of the dual abelian variety $\what X$ (see \cite[\S 11]{Milne86}).
For any fixed prime $\ell \neq \Char(\bk)$, let
\[
T_\ell(X) \coloneqq \varprojlim_{n} X[\ell^n](\bk)
\]
denote the $\ell$-adic Tate module of $X$, which is a free $\bZ_\ell$-module of rank $2g$.
Call
\[
V_\ell(X) \coloneqq T_\ell(X) \otimes_{\bZ_\ell} \bQ_\ell
\]
the $\ell$-adic Tate space of $X$.
It is well known that there is an injective homomorphism
\[
T_\ell \colon \End(X) \otimes_\bZ \bZ_\ell \longinjmap \End_{\bZ_\ell}(T_\ell(X))
\]
of $\bZ_\ell$-algebras and an injective homomorphism
\begin{equation*}
V_\ell \colon \End^\circ(X) \otimes_\bQ \bQ_\ell \longinjmap \End_{\bQ_\ell}(V_\ell(X)) = \Mat_{2g}(\bQ_\ell)
\end{equation*}
of $\bQ_\ell$-algebras; see \cite[\S 19]{Mumford} or \cite[\S 12]{Milne86}.

For a homomorphism $\phi \colon X \to Y$ of abelian varieties, its {\it degree} $\deg \phi$ is defined to be the order of the kernel $\ker \phi$, if it is finite, and $0$ otherwise.
We can extend this definition to any $\alpha \in \End^\circ(X)$ by setting $\deg \alpha = n^{-2g} \deg(n\alpha)$, if $n\alpha \in \End(X)$ for some $n\in\bZ_{>0}$.

\begin{theorem}[{cf.~\cite[\S19, Theorem~4]{Mumford} and \cite[\S12, Proposition~12.9]{Milne86}}]
\label{thm:char-poly}
For any $\alpha \in \End^\circ(X)$, there exists a monic polynomial $P_\alpha(t) \in \bQ[t]$ of degree $2g$ such that $P_\alpha(n) = \deg([n]_X - \alpha)$ for any $n\in\bZ$.
Moreover, if $\alpha\in \End(X)$, then $P_\alpha(t)\in \bZ[t]$; $P_\alpha(t)$ is also equal to the characteristic polynomial $\det(t \, \mathrm{I}_{2g} - V_\ell(\alpha))$ of the induced map $V_\ell(\alpha)$ on $V_\ell(X)$.
\end{theorem}

\begin{definition}
\label{def:char-poly}
The above $P_\alpha(t)\in\bQ[t]$ is called the {\it characteristic polynomial} of $\alpha\in\End^\circ(X)$.
\end{definition}

The Riemann--Roch theorem for abelian varieties takes the following simple form.

\begin{theorem}[{cf.~\cite[\S 16]{Mumford}}]
\label{thm:RR}
For any line bundle $\sL = \sO_X(D)$ on $X$, we have
\[
\chi(\sL) = \frac{D^g}{g!} \quad \text{and} \quad \chi(\sL)^2 = \deg(\phi_{\sL}).
\]
\end{theorem}

We now describe the structure theorem of endomorphism algebras.
By Poincar\'e's complete reducibility theorem (see, e.g., \cite[\S19, Theorem~1]{Mumford}), $X$ is isogenous to a product $X_1 \times \cdots \times X_s$, where each $X_i = A_i^{n_i}$ is isotypic, the $A_i$ are mutually non-isogenous simple abelian varieties, and
\[
\End^\circ(X) \isom \prod_{i=1}^{s} \End^\circ(X_i) \isom \prod_{i=1}^{s} \Mat_{n_i}(\End^\circ(A_i)).
\]
So it suffices to consider a simple abelian variety $A$ over $\bk$.
In this case, $D\coloneqq \End^\circ(A)$ is a division algebra.
Let $K$ denote the center of $D$ which is a field and $K_0$ the maximal totally real subfield of $K$. Set
\begin{equation}
\label{eq:ed}
d^2 = [D:K], \ e = [K:\bQ], \ \text{ and } \ e_0 = [K_0:\bQ].
\end{equation}
Below is Albert's classification of the endomorphism $\bQ$-algebras of simple abelian varieties (see, e.g., \cite[\S21, Application~I, Theorem~2]{Mumford}):
\begin{enumerate}[label=,itemindent=-2em]
\item Type I$(e)$: $d=1$, $e=e_0$ and $D=K=K_0$ is a totally real number field.
\item Type II$(e)$: $d=2$, $e=e_0$, $D$ is an indefinite quaternion division algebra over a totally real number field $K=K_0$ such that $D$ splits at each real place of $K$.
\item Type III$(e)$: $d=2$, $e=e_0$, $D$ is a definite quaternion division algebra over a totally real number field $K=K_0$ such that $D$ does not split at any real place of $K$.
\item Type IV$(e_0, d)$: $e=2e_0$ and $D$ is a division algebra over a CM-field $K \supsetneq K_0$.
In this case, there is an isomorphism $D\otimes_\bQ \bR \isom \bigoplus_{i=1}^{e_0} \Mat_{d}(\bC)$.
\end{enumerate}

Now, let us study representations of the semisimple $\bQ$-algebra $\End^\circ(X)$ of endomorphisms of an abelian variety $X$.
For simplicity, let us first consider the case when
\[
X = A^n
\]
is a power of a simple abelian variety $A$ over $\bk$.
Then $\End^\circ(X) = \Mat_n(D)$ is a simple $\bQ$-algebra and $D\coloneqq \End^\circ(A)$ is a division algebra.
We adopt the notation \eqref{eq:ed}.
Let $V_1, \ldots, V_e$ denote the $e$ nonisomorphic irreducible representations of $\End^\circ(X)$ over $\ol{\bQ}$, where each one has degree $dn$.
We then call $V^{\reduced} \coloneqq \bigoplus_{i=1}^{e} V_i$ the {\it reduced representation} of $\End^\circ(X)$.
In particular, for any $\alpha\in \End^\circ(X)$, one has that
\begin{equation}
\label{eq:splitting}
\alpha \otimes_\bQ 1_{\ol{\bQ}} \in \End^\circ(X) \otimes_\bQ \ol{\bQ} = \Mat_{n}(D)\otimes_\bQ \ol{\bQ} \xrightarrow[\ \ \psi\ \ ]{~\isom~} \bigoplus_{i=1}^{e} \Mat_{dn}(\ol\bQ)
\end{equation}
and
\[
\psi(\alpha \otimes_\bQ 1_{\ol\bQ}) \sim \alpha|_{V^{\reduced}}.
\]
On the other hand, the $2g$-dimensional $\ell$-adic Tate space $V_\ell(X)$ is a faithful representation of $\End^\circ(X)\otimes_\bQ \bQ_\ell$.
We thus obtain the following.

\begin{proposition}[{cf.~\cite[Proposition~12.12]{Milne86}}]
\label{prop:Milne}
With notation as above, there is an isomorphism of representations of the semisimple $\ol{\bQ_\ell}$-algebra $\End^\circ(X)\otimes_{\bQ} \ol{\bQ_\ell}$:
\[
V_\ell(X) \otimes_{\bQ_\ell} \ol{\bQ_\ell} \isom (V^{\reduced})^{\oplus m} \otimes_{\ol{\bQ}} \ol{\bQ}_\ell,
\]
where $m \coloneqq 2g/(edn)$.
In particular, for any $\alpha\in \End^\circ(X)$, $V_\ell(\alpha)$ is similar to $(\alpha|_{V^{\reduced}})^{\oplus m}$.
\end{proposition}

So in order to describe the Jordan canonical form of $V_\ell(\alpha)$, we just need to know the one of $\alpha|_{V^{\reduced}}$, or equivalently, of $\psi(\alpha \otimes_\bQ 1_{\ol\bQ})$ for any arbitrary $\psi$ by the Skolem--Noether theorem.
The following two lemmas take care of the crucial cases when $D$ is of Type III and Type IV, respectively.

\begin{lemma}
\label{lemma:Type-III}
Suppose that $D=\End^\circ(A)$ is of Type III$(e)$.
Then $\alpha|_{V^{\reduced}}$ admits a pseudo-analytic decomposition (see \cref{def:pseudo-analytic}) for any $\alpha\in \End^\circ(X)$,
i.e., $\alpha|_{V^{\reduced}}$ is similar to $J\oplus \ol{J}$, where $J\in \Mat_{en}(\ol\bQ)$ is a direct sum of some standard Jordan blocks and $\ol{J}$ is its complex conjugate.
\end{lemma}
\begin{proof}
By assumption, $D$ is a definite quaternion division algebra $\bH$ over a totally real number field $K=K_0$.
As usual, denote by $\{1, \bbi, \bbj, \bbk\}$ the basis of $\bH$.
We embed $\bH$ into $\bH \otimes_K \ol\bQ \isom \Mat_2(\ol\bQ)$ in a standard way.
This further induces the following injective homomorphism of $K$-algebras:
\[
\begin{array}{ccccc}
\Mat_{n}(\bH) & \longinjmap & \Mat_{n}(\bH) \otimes_K \ol\bQ & \xrightarrow{\ \ \isom\ \ } & \Mat_{2n}(\ol\bQ) \\[2pt]
M = M_1 + M_2 \, \bbj & \longmapsto & M \otimes_K 1_{\ol\bQ} & \longmapsto & 
\begin{pmatrix}
M_1 & M_2 \\
-\ol{M_2} & \ol{M_1}
\end{pmatrix}.
\end{array}
\]
We note that any complex matrix in $\Mat_{2n}(\bC)$ of the above form has a Jordan canonical form $J\oplus \ol{J}$ (see, e.g., \cite[Theorem~5.7.1]{Rodman14}).
On the other hand, it is easy to see that
\[
\Mat_{n}(\bH) \otimes_\bQ \ol\bQ = \Mat_{n}(\bH) \otimes_K (K\otimes_\bQ \ol\bQ) \isom \bigoplus_{i=1}^{e} \Mat_{n}(\bH) \otimes_K \ol\bQ \xrightarrow[]{\ \ \isom \ \ } \bigoplus_{i=1}^{e} \Mat_{2n}(\ol\bQ).
\]
As in Formula~\eqref{eq:splitting}, we denote the above composite isomorphism by $\psi$.
We thus obtain that $\psi(\alpha \otimes_\bQ 1_{\ol{\bQ}})$ admits a pseudo-analytic decomposition.
\cref{lemma:Type-III} follows from the Skolem--Noether theorem.
\end{proof}

\begin{lemma}
\label{lemma:Type-IV}
Suppose that $D=\End^\circ(A)$ is of Type IV$(e_0,d)$.
Then $\alpha|_{V^{\reduced}}$ admits a pseudo-analytic decomposition (see \cref{def:pseudo-analytic}) for any $\alpha\in \End^\circ(X)$,
i.e., $\alpha|_{V^{\reduced}}$ is similar to $J\oplus \ol{J}$, where $J\in \Mat_{e_0dn}(\bC)$ is a direct sum of some standard Jordan blocks and $\ol{J}$ is its complex conjugate.
\end{lemma}
\begin{proof}
By assumption, $D$ is a central division algebra over a CM-field $K \supsetneq K_0$.
In this case, we have
\[
D \otimes_{K_0} \bR = D \otimes_{K} (K \otimes_{K_0} \bR) \isom D \otimes_K \bC \xrightarrow[]{\ \isom\ } \Mat_d(\bC).
\]
It thus follows that
\begin{equation}
\label{eq:IV-C/R}
\Mat_n(D) \otimes_\bQ \bR = \Mat_n(D) \otimes_{K_0} (K_0 \otimes_\bQ \bR) \isom \bigoplus_{i=1}^{e_0} \, \Mat_n(D) \otimes_{K_0} \bR \xrightarrow[]{\ \isom\ } \bigoplus_{i=1}^{e_0} \Mat_{dn}(\bC),
\end{equation}
and
\[
\Mat_n(D) \otimes_\bQ \bC = (\Mat_n(D) \otimes_\bQ \bR) \otimes_\bR \bC \xrightarrow[\Psi]{\ \ \isom \ \ } \bigoplus_{i=1}^{e_0} \Mat_{dn}(\bC) \otimes_\bR \bC.
\]
Note that in Formula~\eqref{eq:IV-C/R} we (have to) think of the matrix algebra $\Mat_{dn}(\bC)$ of complex matrices as a semisimple $\bR$-algebra.
We therefore embed $\Mat_{dn}(\bC)$ into $\Mat_{2dn}(\bR)$ in the standard way:
\[
\begin{array}{cccc}
\iota\colon \Mat_{dn}(\bC) & \longinjmap & \Mat_{2dn}(\bR) \\
M & \longmapsto & 
\begin{pmatrix}
\re M & \im M \\
-\im M & \re M
\end{pmatrix},
\end{array}
\]
so that there is an embedding $\Mat_{dn}(\bC) \otimes_\bR \bC \longinjmap \Mat_{2dn}(\bC)$ of $\bC$-algebras, still denoted by $\iota$.
In general, one can check that $\iota(M)$ is similar to the block diagonal matrix $M\oplus \ol{M}$ for any $M\in \Mat_{dn}(\bC)$; see, e.g., \cite[Proof of Lemma 3.5 (3)]{Hu19}.
Putting all together, we have shown that $\iota\circ\Psi(\alpha \otimes_\bQ 1_{\bC})$ is similar to a direct sum of a matrix $\bigoplus_{i=1}^{e_0} M_i$ and its complex conjugate $\bigoplus_{i=1}^{e_0} \ol{M_i}$, where each $M_i \in \Mat_{dn}(\bC)$.
Therefore, $\iota\circ\Psi(\alpha \otimes_\bQ 1_{\bC})$ admits a pseudo-analytic decomposition.
By the Skolem--Noether theorem, $\psi(\alpha \otimes_\bQ 1_{\ol\bQ})$ is similar to $\iota\circ\Psi(\alpha \otimes_\bQ 1_{\bC})$ for any $\psi\colon \End^\circ(X) \otimes_\bQ \ol\bQ \xrightarrow[]{\ \isom \ } \bigoplus_{i=1}^{e} \Mat_{dn}(\ol\bQ)$.
\cref{lemma:Type-IV} thus follows.
\end{proof}

\begin{proof}[Proof of Theorem \ref{thm:A}]
By Poincar\'e's complete reducibility theorem (see, e.g., \cite[\S19, Theorem~1]{Mumford}), we may assume that $X$ is isotypic, i.e., isogenous to $A^n$ for some simple abelian variety $A$ over $\bk$.
Let $D$ denote the division algebra $\End^\circ(A)$, $K$ the center of $D$, and $K_0$ the maximal totally real subfield of $K$. 
As usual, set
\begin{equation*}
d^2 = [D:K], \ e = [K:\bQ], \ \text{ and } \ e_0 = [K_0:\bQ].
\end{equation*}
By \cref{prop:Milne}, it suffices to show that $(\alpha|_{V^{\reduced}})^{\oplus m}$ admits a pseudo-analytic decomposition (see \cref{def:pseudo-analytic}), where $m \coloneqq 2g/(edn)$.

Suppose that $D$ is of Type I$(e)$ or Type II$(e)$.
Then we have that
\[
\End^\circ(X) \otimes_\bQ \bR = \Mat_n(D) \otimes_\bQ \bR \xrightarrow[\psi]{\ \ \isom\ \ } \bigoplus_{i=1}^{e} \Mat_{dn}(\bR).
\]
Hence $\alpha|_{V^{\reduced}}$ is similar to a real matrix $\psi(\alpha\otimes_\bQ 1_{\bR})$ by the Skolem--Noether theorem.
Note also that in these two cases $m=2g/(edn)$ is an even number.
It follows readily that $(\alpha|_{V^{\reduced}})^{\oplus m}$ admits a pseudo-analytic decomposition.

The case when $D$ is of Type III$(e)$ or Type IV$(e_0,d)$ has been proved in \cref{lemma:Type-III,lemma:Type-IV}.
We thus complete the proof of \cref{thm:A}.
\end{proof}


\section{Proofs of Theorems \ref{thm:B}, \ref{thm:C} and \ref{thm:D}}


We shall prove our main results in the order of \cref{thm:C}, \cref{thm:B}, and \cref{thm:D}.

\begin{proof}[Proof of Theorem \ref{thm:C}]
The equivalence Statement~\eqref{thm:C-1} $\Longleftrightarrow$ Statement~\eqref{thm:C-4} follows from the fact that $V_\ell$ is an injective homomorphism of algebras.
Let $P_\alpha(t)\coloneqq \det(t \, \mathrm{I}_{2g} - V_\ell(\alpha)) \in \bZ[t]$ denote the characteristic polynomial of $\alpha$ (see \cref{def:char-poly}).
Since $H^2_{\et}(X, \bQ_\ell) = \wedge^2 H^1_{\et}(X, \bQ_\ell)$ and the wedge product is preserved by $\alpha^*$, we have that
\begin{align*}
\textrm{Statement~\eqref{thm:C-4}} &\Longleftrightarrow \textrm{all roots of $P_\alpha(t)$ are roots of unity} \\
&\Longleftrightarrow \textrm{all eigenvalues of $\alpha^*|_{H_{\et}^1(X,\bQ_\ell)}$ are roots of unity} \\
&\Longleftrightarrow \textrm{all eigenvalues of $\alpha^*|_{H_{\et}^2(X,\bQ_\ell)}$ are roots of unity} \\
&\Longleftrightarrow \textrm{Statement~\eqref{thm:C-3}},
\end{align*}
where the third implication $\Longleftarrow$ follows from \cref{thm:A}, the fact that for any $\alpha\in \End(X)$ all eigenvalues of $\alpha^*|_{H_{\et}^i(X,\bQ_\ell)}$ are algebraic integers, and Kronecker's theorem.

Finally, we also note that, a priori, for any $\alpha\in \End(X)$, all eigenvalues of $\alpha^*|_{H_{\et}^2(X,\bQ_\ell)}$ and $\alpha^*|_{\N^1(X)_\bQ}$ are algebraic integers since their characteristic polynomials are monic and have integer coefficients.
Further, $\alpha$ is surjective if and only if any one of these induced actions is invertible.
Hence the equivalence Statement~\eqref{thm:C-2} $\Longleftrightarrow$ Statement~\eqref{thm:C-3} follows from \cite[Theorem~1.2]{Hu19} (or \cite[Theorem~1.9]{Hu-lc}) and Kronecker's theorem.
\end{proof}

\begin{remark}
\label{rmk:C-End0}
Instead, let us assume $\alpha\in \End^\circ(X)$ in \cref{thm:C}.
Then it is not hard to have the following implications for the corresponding statements:
\[
\begin{tikzcd}
\textrm{Statement~\eqref{thm:C-1}} \arrow[r, Rightarrow] \arrow[d, Leftrightarrow] & \textrm{Statement~\eqref{thm:C-2}} \arrow[dl, Rightarrow, "?"']                      \\
\textrm{Statement~\eqref{thm:C-4}} \arrow[r, Rightarrow]                       & \textrm{Statement~\eqref{thm:C-3}}. \arrow[u, Rightarrow]
\end{tikzcd}
\]
However, we do not know whether Statement~\eqref{thm:C-2} $\Longrightarrow$ Statement~\eqref{thm:C-4} is true.
In fact, suppose that $n\alpha \eqqcolon \beta \in \End(X)$ for some $n\in \bZ_{>0}$ and $\alpha^*|_{\N^1(X)_\bQ} \coloneqq n^{-2} \beta^*|_{\N^1(X)_\bQ}$ is quasi-unipotent.
Applying \cite[Theorem~1.2]{Hu19} or \cite[Theorem~1.9]{Hu-lc} to $X$ and $\beta$ yields that the spectral radius of $\beta^*|_{H^2_{\et}(X, \bQ_\ell)}$ coincides with the spectral radius of $\beta^*|_{\N^1(X)_\bQ}$, which is equal to $n^2$.
It thus follows that the spectral radius of $\alpha^*|_{H^1_{\et}(X, \bQ_\ell)} \coloneqq n^{-1}\beta^*|_{H^1_{\et}(X, \bQ_\ell)}$ equals $1$.
In other words, all roots of the characteristic polynomial $P_\alpha(t) \coloneqq n^{-2g}P_{\beta}(nt) \in \bQ[t]$ have absolute value $1$.
Because of the absence of Kronecker's theorem for algebraic numbers, we do not know whether all roots of $P_\alpha(t)$ are roots of unity, or equivalently, whether $P_\alpha(t)\in \bZ[t]$.
\end{remark}

\begin{proof}[Proof of Theorem \ref{thm:B}]
By \cref{lemma:zero-entropy}, we see that the pullback $f^*|_{\N^1(X)_\bQ}$ of $f$ on $\N^1(X)_\bQ$ is quasi-unipotent.
Then by \cref{thm:C}, so is $V_\ell(f)$.
It follows from \cref{thm:A} that we can denote the Jordan canonical form of $V_\ell(f)$ by $J \oplus \ol{J}$, where $J = \bigoplus_i J_{\zeta_i,k_i} \in \Mat_{g}(\ol\bQ)$ for some roots of unity $\zeta_i\in \ol\bQ$ and $k_i\in\bZ_{>0}$.
The $k_i$ satisfying $\sum_i k_i = g$ are unique up to permutations.
Replacing $f$ by certain iterate, we may assume that $V_\ell(f)$ is unipotent by \cref{lemma:plov-iterate}, so that its Jordan canonical form could be written as
\begin{equation}
\label{eq:J-decomp}
J = \bigoplus_i J_{1,k_i} \in \Mat_{g}(\bN)
\end{equation}
for the same $k_i\in\bZ_{>0}$.

By \cref{def:plov}, we first have
\[
\plov(f) = \limsup_{n\to \infty} \frac{\log (\Delta_n^g)}{\log n},
\]
where
\[
\Delta_n \coloneqq \Delta_n(f, H_X) \coloneqq \sum_{m=0}^{n-1} (f^m)^*H_X.
\]
Thanks to \cref{lemma:plov}, it suffices to evaluate the degree of $\Delta_n^g$ as a polynomial in $n$.

Using the Riemann--Roch theorem for abelian varieties (see, e.g., \cref{thm:RR}), we have that
\begin{equation}
\label{eq:RR}
\frac{\Delta_n^g}{g!} = \chi(\sO_X(\Delta_n)) \quad \text{and} \quad \chi(\sO_X(\Delta_n))^2 = \deg(\phi_{\sO_X(\Delta_n)}).
\end{equation}
Set $\sL \coloneqq \sO_X(H_X)$.
Note that for any $m$, one has
\[
\phi_{(f^m)^*\!\sL} = \what{f^m} \circ \phi_{\sL} \circ f^m = \phi_{\sL} \circ \phi_{\sL}^{-1} \circ \what{f^m} \circ \phi_{\sL} \circ f^m = \phi_{\sL} \circ (f^m)^\dagger \circ f^m,
\]
where ${}^\dagger$ denotes the Rosati involution on $\End^\circ(X)$.
It follows that 
\begin{align*}
\deg(\phi_{\sO_X(\Delta_n)}) &= \deg(\phi_{\sL \otimes f^*\!\sL \otimes \,\cdots\, \otimes (f^{n-1})^*\!\sL}) \\
&= \deg(\phi_\sL + \phi_{f^*\!\sL} + \cdots + \phi_{(f^{n-1})^*\!\sL}) \\
&= \deg(\phi_\sL + \phi_{\sL} \circ f^\dagger \circ f + \cdots + \phi_{\sL} \circ (f^{n-1})^\dagger \circ f^{n-1}) \\
&= \deg(\phi_\sL) \cdot \deg(\id_X + f^\dagger \circ f + \cdots + (f^{n-1})^\dagger \circ f^{n-1}) \\
&= \deg(\phi_\sL) \cdot \det(V_\ell(\id_X + f^\dagger \circ f + \cdots + (f^{n-1})^\dagger \circ f^{n-1})) \\
&= \deg(\phi_\sL) \cdot \det\Bigg(\sum_{m=0}^{n-1} \, V_\ell((f^m)^\dagger) \cdot V_\ell(f^m)\Bigg).
\end{align*}
For the second last equality, see \cref{thm:char-poly}; the last equality follows from the fact that $V_\ell$ is a homomorphism of $\bQ_\ell$-algebras.
Note that under some fixed isomorphism $\End^\circ(X) \otimes_\bQ \ol{\bQ} \isom \Mat_{N}(\ol{\bQ})$, the Rosati involution becomes complex conjugation so that $V_\ell((f^m)^\dagger) = V_\ell(f^m)^*$ for any $m$ (see \cite[\S21, Application~I, Theorem~2]{Mumford} and \cref{prop:Milne}).
To calculate the determinant of the induced action on $V_\ell(X)$, we use the Jordan canonical form $J\oplus J$ of $V_\ell(f)$ with $J$ given in \cref{eq:J-decomp}.
Combining with \cref{prop:A*A} yields that
\begin{align*}
\deg (\deg(\phi_{\sO_X(\Delta_n)})) &= \deg \det\Bigg(\sum_{m=0}^{n-1} \, (V_\ell(f)^m)^* \cdot V_\ell(f)^m\Bigg) \\
&= \deg \det\Bigg(\sum_{m=0}^{n-1} \, (J^m \oplus J^m)^\sT \cdot (J^m \oplus J^m) \Bigg) \\
&= 2\sum_i k_i^2,
\end{align*}
where the outer $\deg$ means the degree of polynomials in $n$.
It thus follows from \cref{eq:RR} that the degree of $\chi(\sO_X(\Delta_n))$ and hence of $\Delta_n^g$ is $\sum_i k_i^2$.
We complete the proof of \cref{thm:B}.
\end{proof}


\begin{remark}
\label{rmk:optimal}
It follows from \cref{thm:B} that the upper bound in \cref{thm:quadratic} is optimal and could be realized by abelian varieties.
In fact, let $f$ be an automorphism of an abelian variety $X$ of dimension $g$ over $\bk$.
When $g=2m$ is even, choose $f$ and $X$ such that the Jordan canonical form of $V_\ell(f)$ is $J_{1,2}^{\oplus g} = J_{1,2}^{\oplus m} \oplus J_{1,2}^{\oplus m}$.
Then $\plov(f) = m \cdot 2^2 = 2g$.
When $g=2m-1$ is odd, choose $f$ and $X$ such that the Jordan canonical form of $V_\ell(f)$ is $J_{1,2}^{\oplus g-1} \oplus J_{1,1}^{\oplus 2} = (J_{1,2}^{\oplus m-1}\oplus J_{1,1})^{\oplus 2}$.
Then $\plov(f) = (m-1)\cdot 2^2 + 1^2 = 2g-1$.
\end{remark}

\begin{proof}[Proof of Theorem \ref{thm:D}]
By the same argument at the beginning of the proof of \cref{thm:B}, we can denote the Jordan canonical form of $V_\ell(f)$, or equivalently, of $f^*|_{H_{\et}^1(X,\bQ_\ell)}$ by $J \oplus J$, where $J = \bigoplus_i J_{1,k_i} \in \Mat_{g}(\bN)$ for some $k_i\in\bZ_{>0}$, since the maximum size of Jordan blocks of $f^*|_{\N^1(X)_\bQ}$ does not change after replacing $f$ by any iterate.
Without loss of generality, we may assume that $k(J) + 1 = \max_i k_i = k_1\geq k_2 \geq \cdots \geq k_i \geq \cdots$.
It suffices to show that $\|(f^n)^*|_{H_{\et}^2(X,\bQ_\ell)}\| \sim n^{2(k_1-1)}$ with respect any matrix norm $\| \cdot \|$, as $n\to \infty$.

Note that $H_{\et}^2(X,\bQ_\ell) = \wedge^2 H_{\et}^1(X,\bQ_\ell)$ and the wedge product is preserved by $f^*$.
So the pullback action $f^*|_{H_{\et}^2(X,\bQ_\ell)}$ could be represented by the second compound matrix $C_2(f^*|_{H_{\et}^1(X,\bQ_\ell)})$ of $f^*|_{H_{\et}^1(X,\bQ_\ell)}$ (see, e.g., \cite[\S 6.6]{KRY09}).
It thus follows that
\begin{align*}
\|(f^n)^*|_{H_{\et}^2(X,\bQ_\ell)}\| &= \|C_2((f^n)^*|_{H_{\et}^1(X,\bQ_\ell)})\| \\
&\sim \|C_2(J^n \oplus J^n)\| \\
&\sim \|C_2(J^n \oplus J^n)\|_{\max} \\
&\sim n^{2(k_1-1)} = n^{2k(J)},
\end{align*}
as $n\to \infty$, where $\| \cdot \|_{\max}$ denotes the (element-wise) max norm of a matrix and the last equivalence follows from the definition of $C_2$.
In fact, the largest $2$-by-$2$ minor of $J^n \oplus J^n$ is the same as the largest one of $J_{1,k_1}^n \oplus J_{1,k_1}^n$, which is equal to $\binom{n}{k_1-1}^2 \sim n^{2(k_1-1)}$.
\end{proof}

\begin{remark}
A similar calculation also shows that $\|(f^n)^*|_{H_{\et}^4(X,\bQ_\ell)}\| = \|C_4((f^n)^*|_{H_{\et}^1(X,\bQ_\ell)})\| \sim \|C_4(J^n \oplus J^n)\| \sim n^{\max\{4(k_1-2), \, 2(k_1+k_2-2)\}} = O(n^{4(g-2)})$, as $n\to \infty$; moreover, we have
\begin{align*}
\|(f^n)^*|_{H_{\et}^{2k}(X,\bQ_\ell)}\| &= O(2k(g-k)), \\
\|(f^n)^*|_{H_{\et}^{2k-1}(X,\bQ_\ell)}\| &= O(k(g-k)+(k-1)(g-k+1)).
\end{align*}
\end{remark}








\section*{Acknowledgments}


I would like to thank Jason Bell, Michel Brion, Jungkai Chen, Laura DeMarco, Keiji Oguiso, Bjorn Poonen, Zinovy Reichstein, Yuri Zarhin, Yishu Zeng, and De-Qi Zhang for stimulating discussions and conversations.
I specifically thank Hsueh-Yung Lin for kindly answering my many questions about their inspiring work \cite{LOZ} and Chen Jiang for providing \cref{prop:P_A^H} which significantly simplifies the proof of \cref{prop:A*A}.
Finally, I am grateful to the referee for his/her helpful suggestions and comments.

\appendix

\section{Growth rate of determinant of certain power sums associated with unipotent matrices}

\begin{center}
{\small \sc By fei hu and chen jiang}
\end{center}

In this appendix, we prove a technical \cref{prop:A*A} which plays a crucial role in the proof of \cref{thm:B}.
We first introduce a notation.
For any {\it unipotent} complex matrix $A\in\Mat_K(\bC)$, i.e., all eigenvalues of $A$ are $1$, and any positive definite Hermitian matrix $H\in \Mat_K(\bC)$, we define a function $P_A^H(n)$ as follows:
\begin{equation}
\label{eq:P_A^H}
P_A^H(n) \coloneqq \det\Bigg(\sum_{m=0}^{n-1} (A^m)^* H A^m \Bigg), \quad n\in \bZ_{>0},
\end{equation}
where ${}^*$ means the conjugate transpose (a.k.a. Hermitian transpose) of a complex matrix.

If we write $A=\mathrm{I}_K+N$, where $\mathrm{I}_K$ denotes the identity matrix and $N$ is a nilpotent matrix (e.g, $N^K=0$),
then $A^m=(\mathrm{I}_K+N)^m=\sum_{i=0}^{K-1}\binom{m}{i}N^i$, whose entries are polynomials in $m$.
It is easy to see that the entries of the matrix $\sum_{m=0}^{n-1} (A^m)^*HA^m$ are all polynomials in $n$ (with coefficients depending on $A$ and $H$), so is its determinant $P_A^H(n)$.

We first prove a key property about the degree of $P_A^H(n)$.

\begin{prop}
\label{prop:P_A^H}
Let $A\in \Mat_K(\bC)$ be a unipotent matrix and $H\in \Mat_K(\bC)$ be a positive definite Hermitian matrix.
Then the following two equivalent assertions hold.

\begin{enumerate}[label=\emph{(\arabic*)}, ref=\arabic*]
\item \label{prop:P_A^H-1} The degree of $P_A^H(n)$ is independent of the choice of $H$, i.e., $\deg P_A^{H'}(n) = \deg P_A^{H}(n)$ for any positive definite Hermitian matrix $H'\in \Mat_K(\bC)$.

\item \label{prop:P_A^H-2} The degree of $P_A^H(n)$ is a similarity invariant of $A$, i.e., $\deg P_A^H(n)=\deg P_B^H(n)$ for any $B\in \Mat_K(\bC)$ similar to $A$.
\end{enumerate}
In particular, we have that $\deg P_A^H(n) = \deg P_A(n)$, where
\begin{equation}
\label{eq:P_A}
P_A(n) \coloneqq P_A^{\,\mathrm{I}_K}(n) = \det\Bigg(\sum_{m=0}^{n-1} (A^m)^*A^m \Bigg), \quad n\in \bZ_{>0}.
\end{equation}
\end{prop}
\begin{proof}
We first prove the assertion~\eqref{prop:P_A^H-1}.
Note that the set of positive definite Hermitian matrices in $\Mat_K(\bC)$ is an open convex cone.
So for any positive definite Hermitian matrix $H'\in \Mat_K(\bC)$, there exists a positive constant $c \gg 1$ such that
\begin{equation}
\label{eq:HH'}
c^{-1}H' < H < cH'.
\end{equation}
Here for two Hermitian matrices $H_1, H_2$, we use $H_1 < H_2$ to denote that $H_2-H_1$ is positive definite.
It is well known that if $H_1 < H_2$ for two positive definite Hermitian matrices, then $\det(H_1)\leq \det(H_2)$; see, e.g., \cite[Theorem~7.8.21 (Minkowski's determinant inequality)]{HJ13}.

It follows from Formula~\eqref{eq:HH'} that $(A^m)^*HA^m < (A^m)^*(cH')A^m$ for any $m\geq 0$ and hence
\[
\sum_{m=0}^{n-1} (A^m)^*HA^m < c\sum_{m=0}^{n-1} (A^m)^*H'A^m,
\]
as positive definite Hermitian matrices.
So by taking the determinants, we have 
\[
P_{A}^H(n)\leq c^KP_{A}^{H'}(n).
\]
By symmetry, one also has $P_{A}^{H'}(n)\leq c^KP_{A}^{H}(n)$.
So as polynomials, $\deg P_{A}^{H'}(n)=\deg P_{A}^{H}(n)$.

Next, we show the implication \eqref{prop:P_A^H-1} $\Longrightarrow$ \eqref{prop:P_A^H-2}.
If $B=SAS^{-1}$ for some invertible matrix $S\in \Mat_K(\bC)$,
then we have that
\begin{align*}
    P_{B}^H(n)= {}&\det\Bigg(\sum_{m=0}^{n-1} (S^{-1})^*(A^m)^* S^*HS A^mS^{-1}\Bigg) \\
    ={}&|\det(S)|^{-2} \cdot \det\Bigg(\sum_{m=0}^{n-1} (A^m)^* S^*HS A^m\Bigg) \\
    ={}&|\det(S)|^{-2} \cdot P_A^{S^*\!HS}(n),
\end{align*}
where $S^*HS$ is still a positive definite Hermitian matrix.
It follows readily from the assertion~\eqref{prop:P_A^H-1} that $\deg P_{B}^H(n)=\deg P_{A}^{S^*\!HS}(n)=\deg P_{A}^H(n)$.

Lastly, we prove the implication \eqref{prop:P_A^H-2} $\Longrightarrow$ \eqref{prop:P_A^H-1}.
It is known that for any positive definite Hermitian matrix $H\in \Mat_K(\bC)$, we can write $H = S^*S$ for some invertible matrix $S\in \Mat_K(\bC)$.
Denote $B\coloneqq SAS^{-1}$.
Then we have that
\begin{align*}
P_{A}^{H}(n) &= \det\Bigg(\sum_{m=0}^{n-1} (A^m)^* S^*S A^m\Bigg) \\
&= |\det(S)|^{2} \cdot \det\Bigg(\sum_{m=0}^{n-1} (S^{-1})^*(A^m)^* S^*S A^mS^{-1}\Bigg) \\
&= |\det(S)|^{2} \cdot \det\Bigg(\sum_{m=0}^{n-1} (B^m)^* B^m\Bigg) = |\det(S)|^{2} \cdot P_B^{\,\mathrm{I}_K}(n).
\end{align*}
It thus follows from the assertion~\eqref{prop:P_A^H-2} that $\deg P_A^H(n) = \deg P_B^{\,\mathrm{I}_K}(n) = \deg P_A^{\,\mathrm{I}_K}(n)$.
\end{proof}



Based on the above proposition, we are able to gives an explicit formula for $\deg P_A(n)$.

\begin{proposition}
\label{prop:A*A}
Let $A\in \Mat_K(\bC)$ be a unipotent matrix such that its Jordan canonical form is $J_A = \bigoplus_{i=1}^{p} J_{1,k_i}\in\Mat_K(\bN)$.
Then we have that
\[
\deg P_A(n) = \deg P_{J_A}(n) = \sum_{i=1}^{p} k_i^2.
\]
\end{proposition}
\begin{proof}
The first equality has been proved in \cref{prop:P_A^H} \eqref{prop:P_A^H-2}.
For the second one, since
\[
P_{J_A}(n) = \det\Bigg(\sum_{m=0}^{n-1} (J_A^m)^\sT J_A^m \Bigg) = \prod_{i=1}^p \det\Bigg(\sum_{m=0}^{n-1} (J_{1,k_i}^m)^\sT J_{1,k_i}^m \Bigg) = \prod_{i=1}^p P_{J_{1,k_i}}(n),
\]
it suffices to consider the case when $A$ has only one Jordan block, say, $J_{1,k}$. In other words, we just need to show that $\deg P_{J_{1,k}}(n) = k^2$.

Write $J_{1,k} = \mathrm{I}_k + J_{0,k}$.
Note first that
\[
(J_{1,k}^m)^\sT J_{1,k}^m = \sum_{i,j=1}^{k} \binom{m}{i-1}\binom{m}{j-1} (J_{0,k}^{i-1})^\sT J_{0,k}^{j-1}.
\]
It is not difficult to check that the $(i,j)$-th entry of $(J_{1,k}^m)^\sT J_{1,k}^m$ is a polynomial in $m$ of degree $i+j-2$ and its leading coefficient is
$
\frac{1}{(i-1)!(j-1)!}.
$
Taking sum over $m$ yields that the $(i,j)$-th entry of $\sum_{m=0}^{n-1} (J_{1,k}^m)^\sT J_{1,k}^m$ is a polynomial in $n$ of degree $i+j-1$ and its leading coefficient is
\[
\frac{1}{(i-1)!(j-1)!(i+j-1)}.
\]
Let $D_n$ be the matrix which collects the highest degree term of each entry of $\sum_{m=0}^{n-1} (J_{1,k}^m)^\sT J_{1,k}^m$, i.e.,
\begin{equation*}
D_n \coloneqq \Bigg( \frac{n^{i+j-1}}{(i-1)!(j-1)!(i+j-1)} \Bigg)_{1\leq i,j\leq k}.
\end{equation*}
It suffices to show that $\det D_n$ has degree $k^2$ with nonzero coefficient.

First, it is easy to see that $\det D_n$ is homogeneous of degree $\sum_{i=1}^k i + \sum_{j=1}^k j - k = k^2$.
The coefficient of $n^{k^2}$ in $\det D_n$ is exactly
\begin{align*}
\det\Bigg(\frac{1}{(i-1)!(j-1)!(i+j-1)} \Bigg)_{1\leq i,j \leq k} &= \frac{1}{\big(\prod_{i=1}^{k-1} i! \big)^2} \det\Bigg(\frac{1}{i+j-1} \Bigg)_{1\leq i,j \leq k} \\
&= \frac{\big(\prod_{i=1}^{k-1} i! \big)^2}{\prod_{i=1}^{2k-1} i!} > 0.
\end{align*}
Here $\big(\frac{1}{i+j-1}\big)_{1\leq i,j \leq k}$ is a Hilbert matrix (which is a special case of Cauchy matrices) and its determinant has been well known since Cauchy; see, e.g., \cite[\S 0.9.12]{HJ13}.
It follows readily that $\deg P_{J_{1,k}}(n) = \deg \det D_n = k^2$.
\end{proof}

\bibliographystyle{amsalpha}
\bibliography{../mybib}

\end{document}